\documentclass[11pt,reqno]{amsart}
\usepackage{amsmath,amsopn,amssymb,amsthm}
\usepackage[cp1251]{inputenc}
\usepackage[english]{babel}
\usepackage[pagebackref,breaklinks=true,colorlinks=true,linkcolor=blue,citecolor=blue,urlcolor=blue]{hyperref}
\usepackage{graphicx}%

\newtheorem{theorem}{Theorem}[section]

\newtheorem{conjecture}[theorem]{Conjecture}
\newtheorem{corollary}[theorem]{Corollary}

\newtheorem{definition}[theorem]{Definition}

\newtheorem{lemma}[theorem]{Lemma}

\renewenvironment{proof}[1][Proof]{\noindent\textbf{#1.} }{\ \rule{0.5em}{0.5em}}

\begin{document}
\title[Algebraic properties of bounded Killing vector fields]{Algebraic properties of bounded \\ Killing vector fields}
\author{Ming Xu \& Yu.\,G. Nikonorov}

\address{Ming Xu \newline
School of Mathematical Sciences \newline
Capital Normal University \newline
Beijing 100048, P. R. China}
\email{mgmgmgxu@163.com}

\address{Yu.\,G. Nikonorov \newline
Southern Mathematical Institute  \newline
of the Vladikavkaz Scientific Centre  \newline
of the Russian Academy of Sciences \newline
Vladikavkaz, Markus str. 22, 362027, Russia}
\email{nikonorov2006@mail.ru}

\date{}
\maketitle

\begin{abstract}
In this paper, we consider a connected Riemannian manifold $M$ where a connected Lie group $G$ acts effectively and isometrically.
Assume $X\in\mathfrak{g}=\mathrm{Lie}(G)$ defines a bounded Killing vector field, we find some crucial algebraic properties of the
decomposition $X=X_r+X_s$ according to a Levi decomposition $\mathfrak{g}=\mathfrak{r}(\mathfrak{g})+\mathfrak{s}$, where $\mathfrak{r}(\mathfrak{g})$
is the radical, and $\mathfrak{s}=\mathfrak{s}_c\oplus\mathfrak{s}_{nc}$ is a Levi subalgebra. The decomposition $X=X_r+X_s$ coincides with the abstract
Jordan decomposition of $X$, and is unique in the sense that it does not depend on the choice of $\mathfrak{s}$. By these properties, we prove that the
eigenvalues of $\mathrm{ad}(X):\mathfrak{g}\rightarrow\mathfrak{g}$ are
all imaginary. Furthermore, when $M=G/H$ is a Riemannian homogeneous space,
we can completely determine all bounded Killing vector fields induced by vectors in $\mathfrak{g}$. We prove that the space of all these bounded Killing
vector fields, or equivalently the space of all bounded vectors in $\mathfrak{g}$ for $G/H$,
is a compact Lie subalgebra, such that its semi-simple part is the ideal  $\mathfrak{c}_{\mathfrak{s}_c}(\mathfrak{r}(\mathfrak{g}))$ of $\mathfrak{g}$, and
its Abelian part is the sum of
$\mathfrak{c}_{\mathfrak{c}(\mathfrak{r}(\mathfrak{g}))}
(\mathfrak{s}_{nc})$ and all
two-dimensional irreducible $\mathrm{ad}(\mathfrak{r}(\mathfrak{g}))$-representations in
 $\mathfrak{c}_{\mathfrak{c}(\mathfrak{n})}(\mathfrak{s}_{nc})$ corresponding to
nonzero imaginary weights, i.e. $\mathbb{R}$-linear functionals
$\lambda:\mathfrak{r}(\mathfrak{g})\rightarrow
\mathfrak{r}(\mathfrak{g})/\mathfrak{n}(\mathfrak{g})
\rightarrow\mathbb{R}\sqrt{-1}$, where $\mathfrak{n}(\mathfrak{g})$ is the nilradical.
\smallskip

\textbf{Mathematics Subject Classification (2010)}: 22E46, 53C20, 53C30.

\textbf{Key words}: bounded Killing vector field; Killing vector field of constant length;
bounded vector for a coset space; Levi decomposition; Levi subalgebra; nilradical; radical
\end{abstract}

\section{Introduction}
In a recent paper \cite{Ni2019}, the second author considered a Riemannian manifold $M$, which permits the
effective isometric action of a Lie group $G$. He studied the Killing vector field of constant length
induced by a vector $X\in\mathfrak{g}=\mathrm{Lie}(G)$, and proved
that the eigenvalues of $\mathrm{ad}(X):\mathfrak{g}\rightarrow\mathfrak{g}$ are all imaginary.
This discovery inspired him to propose the following conjecture (see Conjecture 1 in \cite{Ni2019}).

\begin{conjecture}\label{question 1}
Assume that a semi-simple Lie group $G$ acts effectively and isometrically on a connected Riemannian manifold $M$, and the vector $X\in\mathfrak{g}$ defines a Killing vector field of
constant length. Then $X$ is a compact vector in $\mathfrak{g}$,
i.e. the subalgebra $\mathbb{R}X$ is compactly imbedded in $\mathfrak{g}$.
\end{conjecture}

See Section \ref{subsection-2.1} for the notions of compact
vector and compactly imbedded subalgebra.

Our initial motivation is to prove Conjecture \ref{question 1}. Our approach is different from that
in \cite{Ni2019}, which depends on the Riemannian structure and the constant length condition. Here we only need to assume that the vector $X\in\mathfrak{g}$ defines a bounded Killing vector
field.

Recall that a Killing vector field on a Riemannian manifold is called bounded if its length function with respect to the given
metric is a bounded function.
This condition is relatively weak. For example, any
Killing vector field on a compact Riemannian manifold is bounded. The special case,
Killing vector fields of constant length, is intrinsically related to Clifford--Wolf translations \cite{BN2008,DX2014}. See
\cite{Ni2015,WPX,XW} for some recent progress on this subject.
On the other hand, curvature conditions may provide obstacles or rigidities for bounded Killing vector fields. For example,
on a complete negatively curved Riemannian manifold, bounded
Killing vector field must be zero \cite{Wo1964}. On a complete
non-positively curved Riemannian manifold, a bounded Killing vector field must be parallel \cite{BN2008-2}.

We first prove the following theorem solving Conjecture \ref{question 1}, not only for Killing vector
fields of constant length, but also for all bounded Killing vector fields.

\begin{theorem}\label{theorem 1}
Let $M$ be a connected Riemannian manifold on which a connected semi-simple Lie group acts effectively and isometrically. Assume $X\in\mathfrak{g}$ defines a bounded Killing vector field. Then $X$ is contained in a
compact ideal in $\mathfrak{g}$.
\end{theorem}

As a compact ideal in the semi-simple Lie algebra $\mathfrak{g}$ generates a compact semi-simple subgroup of $G$, we see immediately after Theorem \ref{theorem 1} that $X$ is a compact vector when it is bounded, and hence $\mathrm{ad}(X):\mathfrak{g}\rightarrow\mathfrak{g}$ has only imaginary eigenvalues.

It is then natural to further study
this spectral property of bounded Killing vector fields when $G$ is not semi-simple.
For this purpose, we take a Levi decomposition $\mathfrak{g}=\mathfrak{r}(\mathfrak{g})+\mathfrak{s}$ for $\mathfrak{g}=\mathrm{Lie}(G)$ (see Section \ref{subsection-2.1}), and then we
have the decomposition $X=X_r+X_s$ accordingly.
Applying the argument for Theorem \ref{theorem 1}, some technique in the proof of Lemma 2.3 and Lemma 2.4 in \cite{Wo2017}, and more Lie algebraic discussion from Lemma \ref{lemma 10} and Corollary \ref{lemma 5} (see Lemma 3 in \cite{Ti1964} for similar argument for bounded automorphisms of Lie groups), we prove the following crucial algebraic properties
for the bounded Killing vector field $X$.

\begin{theorem}\label{theorem 2}
Let $M$ be a connected Riemannian manifold on which
the connected Lie group $G$ acts effectively and isometrically. Assume that $X\in\mathfrak{g}$ defines a bounded
Killing vector field, and $X=X_r+X_s$ according to the Levi
decomposition $\mathfrak{g}=\mathfrak{r}(\mathfrak{g})+\mathfrak{s}$, where
$\mathfrak{s}=\mathfrak{s}_c\oplus\mathfrak{s}_{nc}$.
Then we have the following:
\begin{enumerate}
\item The vector $X_s\in\mathfrak{s}$ is contained in
    the compact semi-simple ideal
    $\mathfrak{c}_{\mathfrak{s}_c}(\mathfrak{r}(\mathfrak{g}))$
    of $\mathfrak{g}$;
\item The vector $X_r\in\mathfrak{r}$ is contained in the center $\mathfrak{c}(\mathfrak{n})$ of $\mathfrak{n}$.
\end{enumerate}
\end{theorem}

Here the centralizer $\mathfrak{c}_\mathfrak{a}(\mathfrak{b})$ of the subalgebra $\mathfrak{b}\subset\mathfrak{g}$ in  the subalgebra $\mathfrak{a}\subset\mathfrak{g}$
is defined as $\mathfrak{c}_\mathfrak{a}(\mathfrak{b})=
\{u\in\mathfrak{a}\,|\,[u,\mathfrak{b}]=0\}$.
In particular, the center
$\mathfrak{c}(\mathfrak{a})$ of $\mathfrak{a}\subset\mathfrak{g}$ coincides with~$\mathfrak{c}_\mathfrak{a}(\mathfrak{a})$.
\smallskip

Theorem \ref{theorem 2} helps us find more algebraic properties for bounded Killing vector fields. In particular, $X=X_r+X_s$ is
an abstract Jordan decomposition which is irrelevant to the choice of the Levi subalgebra $\mathfrak{s}$, and the eigenvalues of $\mathrm{ad}(X)$
coincide with those of $\mathrm{ad}(X_s)$, which are all imaginary (see Theorem \ref{main-cor}). As a direct corollary, we have proved the following spectral property.

\begin{corollary}
Let $M$ be a connected Riemannian manifold on which the connected
Lie group $G$ acts effectively and isometrically. Assume that $X\in\mathfrak{g}$ defines a bounded Killing vector field. Then all eigenvalues of
$\mathrm{ad}(X):\mathfrak{g}\rightarrow\mathfrak{g}$ are imaginary.
\end{corollary}

When $M=G/H$ is a Riemannian homogeneous space on which the connected Lie group $G$ acts effectively, we can apply Theorem \ref{theorem 2}
to prove the following theorem, which completely determine all bounded vectors in $\mathfrak{g}$ for $G/H$, or equivalently all bounded Killing
vector fields induced by vectors in $\mathfrak{g}$
(see Section 2.3 for the notion of bounded vectors for a coset space, and Lemma \ref{lemma 1} for the equivalence).
\begin{theorem}\label{theorem 3}
Let $G/H$ be a Riemannian homogeneous space on which the connected Lie group $G$ acts effectively. Let $\mathfrak{r}(\mathfrak{g})$, $\mathfrak{n}(\mathfrak{g})$ and $\mathfrak{s}=\mathfrak{s}_c\oplus\mathfrak{s}_{nc}$ be the radical, the nilradical, and the Levi subalgebra  respectively.
Then the space of
all bounded vectors in $\mathfrak{g}$ for $G/H$ is a compact subalgebra. Its
semi-simple part coincides with the ideal
$\mathfrak{c}_{\mathfrak{s}_c}(\mathfrak{r}(\mathfrak{g}))$ of~$\mathfrak{g}$,
which is independent of the choice of the Levi subalgebra $\mathfrak{s}$,
and its Abelian part $\mathfrak{v}$
is contained in $\mathfrak{c}(\mathfrak{n}(\mathfrak{g}))$, which coincides with
the sum of $\mathfrak{c}_{\mathfrak{c}(\mathfrak{r}(\mathfrak{g}))}
(\mathfrak{s}_{nc})$ and all two-dimensional
irreducible representations of $\mathrm{ad}(\mathfrak{r}(\mathfrak{g}))$ in
$\mathfrak{c}_{\mathfrak{c}(\mathfrak{n}(\mathfrak{g}))}
(\mathfrak{s}_{nc})$ corresponding to nonzero imaginary weights, i.e. $\mathbb{R}$-linear functionals
$\lambda:\mathfrak{r}\rightarrow\mathfrak{r}/\mathfrak{n}
\rightarrow\mathbb{R}\sqrt{-1}$.
\end{theorem}

Theorem \ref{theorem 3} is a summarization of Theorem \ref{main-cor-2} and Theorem \ref{main-cor-3}.

Note that $\mathfrak{c}_{\mathfrak{s}_{c}}(\mathfrak{r}(\mathfrak{g}))$ is a
compact semi-simple summand in
the Lie algebra direct sum decomposition of $\mathfrak{g}$,
which can be easily determined. For the other, the Abelian factor~$\mathfrak{v}$,
we propose a theoretic algorithm which
explicitly describes all bounded vectors in
$\mathfrak{c}(\mathfrak{n}(\mathfrak{g}))$.
\smallskip

Theorem \ref{theorem 3} provides a simple and self contained proof of the following theorem.

\begin{theorem}\label{theorem 4}
The space of bounded vectors in $\mathfrak{g}$ for a Riemannian
homogeneous space $G/H$ on which the connected Lie group $G$ acts
effectively is irrelevant to the choice of~ $H$.
\end{theorem}

Notice that the arguments in \cite{Ti1964} indicate that the
subset of all bounded isometries in $G$ is irrelevant to the
choice of $H$. So Theorem \ref{theorem 4} can also be
proved by
J. Tits' Theorem 1 in \cite{Ti1964}, which implies that all bounded isometries in $G$ are generated by bounded vectors in $\mathfrak{g}$.
\smallskip

Meanwhile,
Theorem \ref{theorem 3} provides an alternative explanation why in some special cases,
the much stronger constant length condition for Killing vector fields or Clifford--Wolf
condition for translations may be implied by the boundedness condition \cite{MMW,Wo2017}.
\smallskip

At the end, we remark that all lemmas, theorems and corollaries
are still valid when~$M$ is a Finsler manifold. The Finsler metric on a smooth manifold is a natural
generalization of the Riemannian metric, which satisfies the properties of the smoothness, positiveness, homogeneity of degree one and strong convexity,
but not the quadratic property in general. See \cite{BCS2000} for
its precise definition and more details. The proofs for all the results of this work in the Finsler context only
need an add-on from the following well-known fact.
The isometry group of a Finsler manifold is a Lie group~\cite{DH2002} with a compact isotropy subgroup at any point.
\smallskip

This work is organized as following. In Section 2, we summarize some basic knowledge on Lie theory and homogeneous geometry which
are necessary for later discussions. We define the bounded vector in $\mathfrak{g}$ for a smooth coset space $G/H$ and discuss its basic properties and relation to the bounded Killing vector field. In Section 3, we prove
Theorem \ref{theorem 1} and Theorem \ref{theorem 2}. In Section 4, we discuss two applications of Theorem \ref{theorem 2}. One is to prove the Jordan decomposition and spectral properties for bounded Killing vector fields. The other is to study the Lie algebra of all bounded vectors in $\mathfrak{g}$ for a Riemannian homogeneous
space $G/H$, on which $G$ acts effectively. We will provide explicit description for this compact Lie algebra and completely determine all bounded Killing vector fields for a Riemannian homogeneous space.

\newpage

\section{Preliminaries in Lie theory and homogeneous geometry}
\subsection{Some fundamental facts in Lie theory}
\label{subsection-2.1}
Let $\mathfrak{g}$ be a real Lie algebra. Its {\it radical}
$\mathfrak{r}(\mathfrak{g})$ and {\it nilradical} (or {\it nilpotent radical})
$\mathfrak{n}(\mathfrak{g})$ are the unique largest solvable and nilpotent ideals of $\mathfrak{g}$ respectively.
By Corollary 5.4.15 in \cite{JK}, we have
$$
[\mathfrak{r}(\mathfrak{g}),\mathfrak{r}(\mathfrak{g})]\subset
[\mathfrak{r}(\mathfrak{g}),\mathfrak{g}]
\subset\mathfrak{n}(\mathfrak{g})\subset\mathfrak{r}(\mathfrak{g}).
$$

By Levi's theorem,
we can find a semi-simple subalgebra $\mathfrak{s}\subset\mathfrak{g}$, which is the complement of the radical $\mathfrak{r}(\mathfrak{g})$ in $\mathfrak{g}$.
We will further decompose the $\mathfrak{s}$ as
$\mathfrak{s}=\mathfrak{s}_c\oplus\mathfrak{s}_{nc}$,
where~$\mathfrak{s}_c$ and $\mathfrak{s}_{nc}$ are the compact and
noncompact parts of $\mathfrak{s}$
respectively.
We will call
\begin{equation}\label{007}
\mathfrak{g}=\mathfrak{r}(\mathfrak{g})+\mathfrak{s}
\end{equation}
a {\it Levi decomposition}, and the semi-simple subalgebra
$\mathfrak{s}$ in (\ref{007}) a {\it Levi subalgebra}.
By Malcev's Theorem (see Theorem 5.6.13 in \cite{JK}), the Levi subalgebra $\mathfrak{s}$
is unique up to $\mathrm{Ad}\bigl(\exp([\mathfrak{g},\mathfrak{r}(\mathfrak{g})])\bigr)$-actions.

If $G$ is a connected Lie group
with $\mathrm{Lie}(G)=\mathfrak{g}$,
$\mathfrak{r}(\mathfrak{g})$ and $\mathfrak{n}(\mathfrak{g})$ generate closed solvable and nilpotent normal subgroups respectively.

A subalgebra $\mathfrak{k}\subset\mathfrak{g}$ is called {\it compactly imbedded} if after taking closure it
generates a compact subgroup
in the  inner automorphism group $\mathrm{Inn}(\mathfrak{g})=G/Z(G)$.
A vector
$X\in\mathfrak{g}$ is called a {\it compact vector} if $\mathbb{R}X$ is a
compactly imbedded subalgebra of $\mathfrak{g}$.

Assume that $G$ is a connected Lie group with $\mathrm{Lie}(G)=\mathfrak{g}$, and $H$ the connected subgroup
generated by a subalgebra $\mathfrak{h}$.
Obviously if $H$ is compact, then any subalgebra of $\mathfrak{h}$ is compactly imbedded.
The converse statement is not true in general. We call a subgroup $H$ of $G$ {\it compactly imbedded}
if the closure of $\mathrm{Ad}_\mathfrak{g}(H)\subset \mathrm{Aut}(\mathfrak{g})$ is compact.

Any compactly imbedded
subalgebra $\mathfrak{h}$ of $\mathfrak{g}$ is contained in
 a {\it maximal compactly imbedded} subalgebra.
A maximal compactly imbedded subalgebra can be presented as the pre-image in $\mathfrak{g}$ for the subalgebra of
$\mathfrak{g}/\mathfrak{c}(\mathfrak{g})$ generating a
maximal compact subgroup in $G/Z(G)$.
As an immediate corollary for the conjugation theorem for maximal
compact connected subgroups (see Theorem 14.1.3 in \cite{JK}), the maximal compactly imbedded subalgebra is unique up to
$\mathrm{Ad}(G)$-actions.

\subsection{Homogeneous metric and reductive decomposition}

Let $M$ be a Riemannian homogeneous space on which the connected Lie group $G$ acts effectively and isometrically. The effectiveness implies
that $G$ is a subgroup of the isometry group $I(M)$.
When $G$ is a closed subgroup of $I(M)$, then $H$ is compact.
When $G$ is not closed in~$I(M)$, then we still have the following consequence from the discussion in
\cite{MM1988}.

\begin{lemma}\label{lemma -1}
Let $M$ be a Riemannian homogeneous space on
which the connected Lie group $G$ acts effectively and isometrically. Then the isotropy subgroup $H$
at any $x\in M$ and its Lie algebra $\mathfrak{h}$
are compactly imbedded.
\end{lemma}

To be more self contained, we propose a direct proof here.

\begin{proof}
Let $\overline{G}$ be the closure of $G$ in $I(M)$ and $\overline{H}$ be the isotropy subgroup at $x\in M$ for
the $\overline{G}$-action on $M$. Then $\overline{H}$ is
compact. On the other hand, the property that
$\mathrm{Ad}(G)$-actions preserve $\mathfrak{g}$ can be
passed by continuity to $\overline{G}$, i.e. $\mathfrak{g}$
is an ideal of $\overline{\mathfrak{g}}=\mathrm{Lie}(\overline{G})$.
Denote $\mathrm{Ad}_\mathfrak{g}$ the restriction of $\mathrm{Ad}(\overline{G})$-actions from $\mathfrak{g}$ to
$\mathfrak{g}$, then the subgroup $\mathrm{Ad}_\mathfrak{g}(H)$ of~$\mathrm{Aut}(\mathfrak{g})$
(which is contained in $\mathrm{Inn}(\mathfrak{g})$ because of the connectedness of $G$) is contained
in the compact subgroup $\mathrm{Ad}_\mathfrak{g}(\overline{H})$.
From this argument, we also see that both $H$ and $\mathfrak{h}$ are compactly
imbedded.
\end{proof}
\smallskip

Now we further assume $M=G/H$ is a Riemannian homogeneous space.

The $H$-action on $T_o(G/H)$ at $o=eH$ is called the {\it isotropy action}.
A linear direct sum decomposition $\mathfrak{g}=\mathfrak{h}+\mathfrak{m}$, where $\mathrm{Lie}(H)=\mathfrak{h}$, is called a {\it reductive decomposition} for $G/H$, if it is $\mathrm{Ad}(H)$-invariant.
We can identify $T_o(G/H)$ with $\mathfrak{m}$ such that the isotropy action coincides with the $\mathrm{Ad}(H)$-action on $\mathfrak{m}$.

Generally speaking, there exist many different reductive decompositions for a Riemannian homogeneous space $G/H$.  A canonical one can be constructed by the following lemma, which summarizes Lemma 2 and Remark 1 in \cite{Ni2017}.

\begin{lemma}\label{lemma 0}
Let $G/H$ be a Riemannian homogeneous space on which $G$ acts
effectively. Then we have the following:
\begin{enumerate}
\item The restriction of the Killing form $B_\mathfrak{g}$ of $\mathfrak{g}$ to $\mathfrak{h}$ is negative definite;
\item The $B_\mathfrak{g}$-orthogonal decomposition $\mathfrak{g}=\mathfrak{h}+\mathfrak{m}$ is reductive and $\mathfrak{n}(\mathfrak{g})\subset\mathfrak{m}$.
\end{enumerate}
\end{lemma}



\subsection{Bounded vector for a coset space}

For any smooth coset space $G/H$, where $H$ is a closed subgroup of $G$, $\mathrm{Lie}(G)=\mathfrak{g}$ and $\mathrm{Lie}(H)=\mathfrak{h}$, we denote $\mathrm{pr}_{\mathfrak{g}/\mathfrak{h}}$ the natural linear projection from $\mathfrak{g}$ to $\mathfrak{g}/\mathfrak{h}$. We call any vector $X\in\mathfrak{g}$
a {\it bounded vector} for $G/H$, if
\begin{equation}\label{002}
f(g)=\|\mathrm{pr}_{\mathfrak{g}/\mathfrak{h}}\bigl(\mathrm{Ad}(g)X\bigr)\|,
\quad\forall\, g\in G,
\end{equation}
is a bounded function, where $\|\cdot\|$ is any norm on
$\mathfrak{g}/\mathfrak{h}$.

Since $\mathfrak{g}/\mathfrak{h}$ has a finite dimension, any two norms $\|\cdot\|_1$ and $\|\cdot\|_2$ on it are equivalent
in the sense that
$$c_1\|u\|_1\leq \|u\|_2\leq c_2\|u\|_1,\quad\forall\, u\in\mathfrak{g}/\mathfrak{h},$$
where $c_1$ and $c_2$ are some positive constants. So the boundedness of $X\in\mathfrak{g}$ for $G/H$ is not relevant to
the choice of the norm.

When $\|\cdot\|$
is an $\mathrm{Ad}(H)$-invariant quadratic norm, which defines
a $G$-invariant Riemannian metric on $G/H$, the function $f(\cdot)$ on $G$ defined in (\ref{002}) is right $H$-invariant,
so it can be descended to $G/H$, and coincides with the length
function of the Killing vector field induced by $X$.
Summarizing this observation, we have the following lemma.

\begin{lemma}\label{lemma 1}
If $X\in\mathfrak{g}$ is a bounded vector for $G/H$, then it defines a bounded Killing vector field for any $G$-invariant Riemannian metric on $G/H$. Conversely, if $G/H$ is endowed with a $G$-invariant Riemannian metric and $X\in\mathfrak{g}$ induces
a bounded Killing vector field, then $X$ is a bounded vector
for $G/H$.
\end{lemma}

The boundedness condition may be kept when we change the coset space. By definition, it is obvious to see

\begin{lemma}\label{lemma 2}
A vector $X\in\mathfrak{g}$ is bounded for the smooth coset space
$G/H$ iff it is bounded for the universal covering $\widetilde{G}/\widetilde{H}$ of $G/H$, where $\widetilde{G}$ is the universal
covering group of $G$, and $\widetilde{H}$ is closed connected subgroup which $\mathrm{Lie}(H)=\mathfrak{h}$ generates in $\widetilde{G}$.
\end{lemma}

For any chain of subalgebras $\mathfrak{h}\subset\mathfrak{k}\subset\mathfrak{g}$, the natural linear
projection $\mathrm{pr}:\mathfrak{g}/\mathfrak{h}\rightarrow
\mathfrak{g}/\mathfrak{k}$ is continuous with respect to standard
topologies. So it maps bounded sets to bounded sets, with respect to any norms $\|\cdot\|_1$ and $\|\cdot\|_2$ on
$\mathfrak{g}/\mathfrak{h}$ and $\mathfrak{g}/\mathfrak{k}$ respectively. Obviously
$\mathrm{pr}\circ\mathrm{pr}_{\mathfrak{g}/\mathfrak{h}}=
\mathrm{pr}_{\mathfrak{g}/\mathfrak{k}}$, so
$$
\mathrm{pr}\bigl(\mathrm{pr}_{\mathfrak{g}/\mathfrak{h}}
(\mathrm{Ad}(G)X)\bigr)=\mathrm{pr}_{\mathfrak{g}/\mathfrak{k}}
(\mathrm{Ad}(G)X).
$$
By these observations, it is easy to prove the following lemma.

\begin{lemma}\label{lemma 3}
Assume $K$ is a closed subgroup of $G$ which Lie algebra $\mathfrak{k}$ satisfies  \linebreak
$\mathfrak{h}\subset\mathfrak{k}\subset\mathfrak{g}$. If $X\in\mathfrak{g}$ is bounded for $G/H$, then it is bounded for $G/K$ as well.
\end{lemma}

To summarize, the boundedness of Lie algebra vectors for a coset space is originated and intrinsically related to the boundedness of Killing vector fields for a homogeneous metric. However it is an algebraic condition,
which can be discussed more generally and is not relevant to the
choice or existence of homogeneous metrics.

\section{Proof of Theorem \ref{theorem 1} and Theorem \ref{theorem 2}}

\subsection{A key lemma for proving Theorem \ref{theorem 2}}

\begin{lemma}\label{lemma 10}
Let $\mathfrak{g}$ be a Lie algebra and $\mathfrak{g}=\mathfrak{r}(\mathfrak{g})+\mathfrak{s}$ be
a Levi decomposition. Then we have the following Lie algebra direct sum for the centralizer
$\mathfrak{c}_{\mathfrak{g}}(\mathfrak{n}(\mathfrak{g}))$ of the nilradical
$\mathfrak{n}(\mathfrak{g})$ in $\mathfrak{g}$\,{\rm:}
\begin{equation}\label{011}
\mathfrak{c}_{\mathfrak{g}}({\mathfrak{n}(\mathfrak{g})})=
\bigl(\mathfrak{c}_{\mathfrak{g}}({\mathfrak{n}(\mathfrak{g})})
\cap\mathfrak{r}(\mathfrak{g})\bigr)\oplus
\bigl(\mathfrak{c}_{\mathfrak{g}}
({\mathfrak{n}(\mathfrak{g})})\cap\mathfrak{s}\bigr)
=\mathfrak{c}_{\mathfrak{r}(\mathfrak{g})}
(\mathfrak{n}(\mathfrak{g}))
\oplus\mathfrak{c}_\mathfrak{s}(\mathfrak{n}(\mathfrak{g})).
\end{equation}
Moreover we have the following\,{\rm:}
\begin{enumerate}
\item The two summands $\mathfrak{c}_{\mathfrak{r}(\mathfrak{g})}
    (\mathfrak{n}(\mathfrak{g}))
=\mathfrak{c}(\mathfrak{n}(\mathfrak{g}))$ and
$\mathfrak{c}_\mathfrak{s}(\mathfrak{n}(\mathfrak{g}))=
\mathfrak{c}_\mathfrak{s}({\mathfrak{r}(\mathfrak{g})})$ are Abelian and semi-simple ideals of $\mathfrak{g}$ respectively.
\item The summand $\mathfrak{c}_\mathfrak{s}(\mathfrak{r}(\mathfrak{g}))$ is
contained in the intersection of all Levi subalgebras, so it does not depend on the choice of the Levi subalgebra $\mathfrak{s}$.
\end{enumerate}
\end{lemma}

\begin{proof}
Firstly, we prove (\ref{011}) as a linear decomposition.

Assume conversely that this is not true, then we can find a vector $X\in\mathfrak{g}$ such that
$[X,\mathfrak{n}(\mathfrak{g})]=0$ and $[X_s,\mathfrak{n}(\mathfrak{g})]\neq 0$.
Denote $\mathrm{ad}_{\mathfrak{n}(\mathfrak{g})}$ the restriction
of the $\mathrm{ad}$-action from $\mathfrak{n}(\mathfrak{g})$ to
$\mathfrak{n}(\mathfrak{g})$. Then
$\mathrm{ad}_{\mathfrak{n}(\mathfrak{g})}(X_r)
=-\mathrm{ad}_{\mathfrak{n}(\mathfrak{g})}(X_s)$ is a nonzero linear endomorphism in
the general linear Lie algebra $\mathfrak{gl}(\mathfrak{n}(\mathfrak{g}))=\mathrm{Lie}(\mathrm{GL}(\mathfrak{n}(\mathfrak{g})))$ where $\mathfrak{n}(\mathfrak{g})$
as well as its subspaces are viewed as real vector spaces.

The map $\mathrm{ad}_{\mathfrak{n}(\mathfrak{g})}$
is a Lie algebra endomorphism
from $\mathfrak{g}$ to $\mathfrak{gl}(\mathfrak{n}(\mathfrak{g}))$.
The vector $X_s$ generates a semi-simple ideal $\mathfrak{s}_1$
of $\mathfrak{s}$, which can be presented as
\begin{equation}\label{004}
\mathfrak{s}_1=\mathbb{R}X_s+
[\mathfrak{s},X_s]+[\mathfrak{s},[\mathfrak{s},X_s]]+
[\mathfrak{s},[\mathfrak{s},[\mathfrak{s},X_s]]]+\cdots.
\end{equation}
Meanwhile, $X_r$ generates a sub-representation space $\mathfrak{v}_1$ in $\mathfrak{r}(\mathfrak{g})$
for the $\mathrm{ad}(\mathfrak{s})$-actions, i.e.
\begin{equation}\label{005}
\mathfrak{v}_1=\mathbb{R}X_r+
[\mathfrak{s},X_r]+[\mathfrak{s},[\mathfrak{s},X_r]]+
[\mathfrak{s},[\mathfrak{s},[\mathfrak{s},X_r]]]+\cdots.
\end{equation}

Compare (\ref{004}) and (\ref{005}), we can see
that $\mathrm{ad}_{\mathfrak{n}(\mathfrak{g})}
(\mathfrak{s}_1)$ and
$\mathrm{ad}_{\mathfrak{n}(\mathfrak{g})}
(\mathfrak{v}_1)$ have the same image in $\mathfrak{gl}(\mathfrak{n}(\mathfrak{g}))$, i.e.
$$\mathrm{ad}_{\mathfrak{n}(\mathfrak{g})}
(\mathfrak{s}_1)=\mathrm{ad}_{\mathfrak{n}(\mathfrak{g})}
(\mathfrak{v}_1)
=\mathbb{R}A+
[\mathrm{ad}_{\mathfrak{n}(\mathfrak{g})}(\mathfrak{s}),A]+
[\mathrm{ad}_{\mathfrak{n}(\mathfrak{g})}(\mathfrak{s}),
[\mathrm{ad}_{\mathfrak{n}(\mathfrak{g})}(\mathfrak{s}),A]]
+\cdots.$$

Denote
$\mathfrak{u}_1=\mathrm{ad}_{\mathfrak{n}(\mathfrak{g})}
(\mathfrak{s}_1)$
and $\mathfrak{u}_2=\mathrm{ad}_{\mathfrak{n}(\mathfrak{g})}
(\mathfrak{r}(\mathfrak{g}))$.
We have just showed $0\neq\mathfrak{u}_1\subset\mathfrak{u}_2$.
Since $\mathrm{ad}_{\mathfrak{n}(\mathfrak{g})}:\mathfrak{g}\rightarrow \mathfrak{gl}(\mathfrak{n}(\mathfrak{g}))$
is a Lie algebra endomorphism,
$\mathfrak{u}_1$ is semi-simple and $\mathfrak{u}_2$ is solvable.
But this is impossible, so (\ref{011}) is a linear direct sum decomposition.
\smallskip

Further, we prove that $\mathfrak{c}_{\mathfrak{r}(\mathfrak{g})}(\mathfrak{n}(\mathfrak{g}))
=\mathfrak{c}(\mathfrak{n}(\mathfrak{g}))$ is an Abelian ideal of $\mathfrak{g}$.

The summand
$\mathfrak{c}_{\mathfrak{r}(\mathfrak{g})}
(\mathfrak{n}(\mathfrak{g}))$ in (\ref{011}) is an ideal of
$\mathfrak{g}$ contained in the radical $\mathfrak{r}(\mathfrak{g})$. It is not hard to
check that $\mathfrak{c}_{\mathfrak{r}(\mathfrak{g})}
(\mathfrak{n}(\mathfrak{g}))
+\mathfrak{n}(\mathfrak{g})$
is a nilpotent ideal of $\mathfrak{g}$. By the definition of the nilradical, we must have
$\mathfrak{c}_{\mathfrak{r}(\mathfrak{g})}
(\mathfrak{n}(\mathfrak{g}))
\subset\mathfrak{n}(\mathfrak{g})$, i.e. $\mathfrak{c}_{\mathfrak{r}(\mathfrak{g})}
(\mathfrak{n}(\mathfrak{g}))=\mathfrak{c}(\mathfrak{n}(\mathfrak{g}))$.
So it is an Abelian ideal.

Finally, we prove that $\mathfrak{c}_\mathfrak{s}(\mathfrak{n}(\mathfrak{g}))=
\mathfrak{c}_\mathfrak{s}(\mathfrak{r}(\mathfrak{g}))$ is a semi-simple ideal of $\mathfrak{g}$
contained in the intersection of all Levi subalgebras.

Obviously $\mathfrak{c}_\mathfrak{s}(\mathfrak{n}(\mathfrak{g}))$ is an ideal of $\mathfrak{s}$. It is
a semi-simple Lie algebra itself, so we have $[\mathfrak{c}_\mathfrak{s}(\mathfrak{n}(\mathfrak{g}))
,\mathfrak{c}_\mathfrak{s}(\mathfrak{n}(\mathfrak{g}))]
=\mathfrak{c}_\mathfrak{s}(\mathfrak{n}(\mathfrak{g}))$. It commutes with $\mathfrak{r}(\mathfrak{g})$
because
\begin{eqnarray*}
[\mathfrak{r}(\mathfrak{g}),\mathfrak{c}_\mathfrak{s}(\mathfrak{n}(\mathfrak{g}))]
&\subset&[\mathfrak{r}(\mathfrak{g}),[\mathfrak{c}_\mathfrak{s}(\mathfrak{n}(\mathfrak{g})),
\mathfrak{c}_\mathfrak{s}(\mathfrak{n}(\mathfrak{g}))]]
=[[\mathfrak{r}(\mathfrak{g}),\mathfrak{c}_\mathfrak{s}(\mathfrak{n}(\mathfrak{g}))],
\mathfrak{c}_\mathfrak{s}(\mathfrak{n}(\mathfrak{g}))]\\
&\subset&
[\mathfrak{n}(\mathfrak{g}),\mathfrak{c}_\mathfrak{s}(\mathfrak{n}(\mathfrak{g}))]=0.
\end{eqnarray*}
So we get $\mathfrak{c}_\mathfrak{s}(\mathfrak{n}(\mathfrak{g}))=
\mathfrak{c}_\mathfrak{s}(\mathfrak{r}(\mathfrak{g}))$.

It is an ideal of $\mathfrak{g}$ because
$
[\mathfrak{g},\mathfrak{c}_\mathfrak{s}(\mathfrak{r}(\mathfrak{g}))]
=[\mathfrak{s},\mathfrak{c}_\mathfrak{s}
(\mathfrak{r}(\mathfrak{g}))]\subset\mathfrak{c}_\mathfrak{s}
(\mathfrak{r}(\mathfrak{g}))$.
Therefore, we have $\mathfrak{c}_\mathfrak{s}(\mathfrak{r}(\mathfrak{g}))=
\mathrm{Ad}(g)\mathfrak{c}_\mathfrak{s}(\mathfrak{r}(\mathfrak{g}))
\subset\mathrm{Ad}(g)\mathfrak{s}$ for all $g\in G$.
So by Malcev's Theorem, $\mathfrak{c}_\mathfrak{s}(\mathfrak{r}(\mathfrak{g}))$
is contained in all Levi subalgebras, and thus independent of the choice of the Levi subalgebra.

We have proved all statements and finished the proof of the lemma.
\end{proof}
\smallskip

By similar arguments as above, we can also establish the Lie algebra direct sum
$\mathfrak{c}_\mathfrak{s}(\mathfrak{r}(\mathfrak{g}))=
\mathfrak{c}_{\mathfrak{s}_{c}}(\mathfrak{r}(\mathfrak{g}))\oplus
\mathfrak{c}_{\mathfrak{s}_{nc}}(\mathfrak{r}(\mathfrak{g}))$ in which
each summand is a semi-simple ideal of $\mathfrak{g}$. So we get the following
corollary.

\begin{corollary}\label{lemma 5}
Keep all relevant notations and assumptions, then we have
the following Lie algebra direct sum decomposition,
$$\mathfrak{c}_\mathfrak{g}(\mathfrak{n}(\mathfrak{g}))
=\mathfrak{c}_{\mathfrak{s}_{c}}(\mathfrak{r}(\mathfrak{g}))\oplus
\mathfrak{c}_{\mathfrak{s}_{nc}}(\mathfrak{r}(\mathfrak{g}))
\oplus\mathfrak{c}(\mathfrak{n}(\mathfrak{g})),$$
in which each summand is an ideal of $\mathfrak{g}$.
\end{corollary}

\subsection{Proof of Theorem \ref{theorem 1}}

Fix any $x\in M$, and denote $H$ the isotropy subgroup of $G$ at $x$. The
smooth coset space $G/H$ can be identified with an immersed submanifold in $M$. The submanifold metric on $G/H$ is $G$-invariant. The restriction of the bounded Killing vector field
induced by $X\in\mathfrak{g}$ to $G\cdot x=G/H$ is still a bounded Killing vector
field, induced by the same $X$. By Lemma \ref{lemma 1}, $X\in\mathfrak{g}$ is a bounded vector for $G/H$.

By Lemma \ref{lemma -1}, $\mathfrak{h}$ is compactly imbedded.
We can find a maximal compactly imbedded subalgebra $\mathfrak{k}$ of $\mathfrak{g}$ such that
$\mathfrak{h}\subset\mathfrak{k}\subset\mathfrak{g}$.
Denote $\widetilde{G}$ the universal cover of $G$, $\widetilde{H}$, and~$\widetilde{K}$ the connected subgroup of $\widetilde{G}$ generated by
$\mathfrak{h}$ and $\mathfrak{k}$ respectively. The subgroup~$\widetilde{K}$ is closed because it is the identity component of the
pre-image in $\widetilde{G}$ for a maximal compact subgroup of $G/Z(G)$.
By Lemma \ref{lemma 2} and
Lemma \ref{lemma 3}, $X\in\mathfrak{g}$ is also bounded for
$\widetilde{G}/\widetilde{H}$ and $\widetilde{G}/\widetilde{K}$.

Since $G$ is semi-simple,
we have $\widetilde{G}=\widetilde{G}_c\times\widetilde{G}_{nc}$, where
$\widetilde{G}_c$ and~$\widetilde{G}_{nc}$ are the compact and non-compact parts of $\widetilde{G}$ respectively,
$\widetilde{K}=\widetilde{G}_c\times\widetilde{K}_{nc}$, and
$X=X_c+X_{nc}$ with $X_c\in\mathfrak{g}_c=\mathrm{Lie}(\widetilde{G}_c)$
and $X_{nc}\in\mathfrak{g}_{nc}=\mathrm{Lie}(\widetilde{G}_{nc})$ accordingly. The coset space
$\widetilde{G}/\widetilde{K}=\widetilde{G}_{nc}/\widetilde{K}_{nc}$
is a symmetric space of non-compact type. The vector $X_{nc}\in\mathfrak{g}_{nc}$ defines the
same Killing vector field as $X$ on $\widetilde{G}/\widetilde{K}$, so
it is bounded as well. Since the Riemannian symmetric metric on $\widetilde{G}/\widetilde{K}$ has negative Ricci curvature and
non-positive sectional curvature, the bounded vector $X_{nc}$ must vanish \cite{Wo1964}. So $X=X_c$ is contained in the compact ideal
$\mathfrak{g}_c$ in $\mathfrak{g}$.

This completes the proof of Theorem \ref{theorem 1}.

\subsection{Proof of Theorem \ref{theorem 2}}

The key steps are summarized as the following two claims.
\smallskip

{\bf Claim 1:} $X_s$ is contained in a compact ideal of $\mathfrak{s}$.

The proof of Claim 1 applies a similar method as for Theorem \ref{theorem 1}.

By similar argument as in the proof of Theorem \ref{theorem 1},
we can restrict our discussion to any orbit $G\cdot x=G/H$ in $M$,
where the isotropy subgroup $H$ has a compactly imbedded Lie algebra.
The vector $X$ indicated in Theorem \ref{theorem 2} is
bounded for $G/H$.

The radical $\mathfrak{r}(\mathfrak{g})$ generates a closed normal subgroup $R$ of ${G}$ and its product $RH$ with the compact subgroup ${H}$ is also a closed subgroup. By Lemma \ref{lemma 3}, $X\in\mathfrak{g}$ is bounded for $G/HR$. We can identify $G/HR$
as the orbit space for the left $R$-actions on $G/H$. So
$G/HR$ admits a $G$-invariant metric
induced by submersion. On the other hand, the coset space $G/HR$
can be identified as $S/H_S=(G/R)/(HR/R)$, where the Lie algebra of $S=G/R$ can be identified with $\mathfrak{s}$ by Levi's Theorem,
and $\mathrm{Lie}(H_S)$ is a compactly imbedded subgroup because it is the image of
the compactly imbedded $\mathfrak{h}$ in~$\mathfrak{g}/\mathfrak{r}(\mathfrak{g})$. With this identification, $X$ defines
the same Killing vector field as $X_s$ on~$S/H_S$. By Lemma \ref{lemma 1}, $X_s$ is bounded for $S/H_S$. Now we have the semi-simpleness for $S$ and the compactly imbedded property for $\mathrm{Lie}(H_S)$, so we can apply a similar argument as for Theorem \ref{theorem 1} to prove $X_s$ is contained in a compact ideal of
$\mathfrak{s}$.

This completes the proof of Claim 1.
\smallskip

{\bf Claim 2:} $X$ commutes with the nilradical $\mathfrak{n}$.

To prove this claim,
we still restrict our discussion to a single $G$-orbit. But we need to be careful because the effectiveness is required in later
discussion. The following lemma guarantees that suitable $G$-orbits with effective $G$-actions can be found.

\begin{lemma}\label{lemma 4}
Let $M$ be a connected Riemannian homogeneous space on which a connected Lie group $G$ acts effectively. Then there
exists $x\in M$, such that $G$ acts effectively on $G\cdot x$.
\end{lemma}

\begin{proof}
Denote $\overline{G}$ the closure of $G$ in $I(M)$. Then
the $\overline{G}$-action on $M$ is proper (see Proposition 3.62 in \cite{AB2015}). By the Principal Orbit Theorem (see Theorem 3.82 in \cite{AB2015}),
the principal orbit
type for the $\overline{G}$-action is unique up to conjugations, and the union $\mathcal{U}$ of all principal orbits is open dense in $M$.

Let $G\cdot x$ be any $G$-orbit in $\mathcal{U}$, and assume $g\in G$ acts trivially on $G\cdot x$. Because $G\cdot x$ is dense in
$\overline{G}\cdot x$, $g$ acts trivially on $\overline{G}\cdot x$
as well. Now we consider any other orbit $\overline{G}\cdot y$ in $\mathcal{U}$. The point $y$ can be suitably chosen such that $x$ and $y$ have the same isotropy subgroups $\overline{G}_x=\overline{G}_y$ in $\overline{G}$. Then their isotropy subgroups in $G$ are the same because $G_x=\overline{G}_x\cap G=\overline{G}_y\cap G=G_y$. The $g$-action on $G\cdot y$ is trivial, and by continuity, that on $\overline{G}\cdot y$ is trivial as well. This argument proves that $g$ acts trivially on the dense open subset $\mathcal{U}$ in $M$, so it acts trivially
on $M$.
Due to the effectiveness of the $G$-action, we must have $g=e\in G$.

To summarize, the $G$-action on $G\cdot x\subset\mathcal{U}$ is effective, which completes the proof of this lemma.
\end{proof}
\smallskip

Take the orbit $G\cdot x=G/H$ indicated in Lemma \ref{lemma 4},
endowed with the invariant submanifold metric.
Since $X\in\mathfrak{g}$ defines a bounded Killing vector field on the whole manifold,
it also defines a bounded Killing vector field when restricted to $G\cdot x$. So by Lemma \ref{lemma 1}, $X$ is
a bounded vector for $G/H$.

By Lemma \ref{lemma 0}, we have $B_\mathfrak{g}$-orthogonal reductive decomposition $\mathfrak{g}=\mathfrak{h}+\mathfrak{m}$ with $\mathfrak{n}(\mathfrak{g})\subset\mathfrak{m}$.
Denote $X_\mathfrak{m}$ the $\mathfrak{m}$-component of $X$. For any $Y\in\mathfrak{n}(\mathfrak{g})$,
\begin{equation}\label{001}
\mathrm{pr}_\mathfrak{m}(\mathrm{Ad}(\exp(tY))X)
=X_\mathfrak{m}+t\,[Y,X]+\frac{t^2}{2!}\,[Y,[Y,X]]+\cdots,
\end{equation}
in which all terms except the first one in the right side are contained in $\mathfrak{n}(\mathfrak{g})$. Since
$\mathfrak{n}(\mathfrak{g})$ is nilpotent, the right side of (\ref{001}) is in fact a vector-valued polynomial with respect to~$t$.
If it has a positive degree, we can get $$\lim_{t\rightarrow\infty}
\|\mathrm{pr}_\mathfrak{m}\bigl(\mathrm{Ad}(\exp(tY))X)\bigr)\|=+\infty,$$
for any norm $\|\cdot\|$ on $\mathfrak{m}$. This is a
contradiction to the boundedness of $X$ for $G/H$. So we get $[X,Y]=0$ for any $Y\in\mathfrak{n}(\mathfrak{g})$ which proves Claim 2.
\smallskip

Finally, we finish the proof of Theorem \ref{theorem 2}.  Claim 2 indicates that
$X\in\mathfrak{c}_\mathfrak{g}(\mathfrak{n}(\mathfrak{g}))$. By Lemma \ref{lemma 10}
or Corollary \ref{lemma 5}, we have $X_r\in\mathfrak{c}_{\mathfrak{r}(\mathfrak{g})}
=\mathfrak{c}(\mathfrak{n}(\mathfrak{g}))$ and $X_s\in
\mathfrak{c}_{\mathfrak{s}}(\mathfrak{n}(\mathfrak{g}))=\mathfrak{c}_{\mathfrak{s}_c}
(\mathfrak{r}(\mathfrak{g}))\oplus\mathfrak{c}_{\mathfrak{s}_{nc}}
(\mathfrak{r}(\mathfrak{g}))$. Claim 1 indicates $X_s$ is contained in the compact semi-simple
ideal $\mathfrak{c}_{\mathfrak{s}_c}
(\mathfrak{r}(\mathfrak{g}))$ of $\mathfrak{g}$.

This finishes the proof of Theorem \ref{theorem 2}.
\section{Applications of Theorem \ref{theorem 2}}
\subsection{Jordan decomposition and spectral property for
bounded Killing vector fields}

Theorem \ref{theorem 2} and Lemma \ref{lemma 10} provide the following obvious observations for $X=X_r+X_s\in\mathfrak{g}$ which defines a bounded
Killing vector field:
\begin{enumerate}
\item The linear endomorphism $\mathrm{ad}(X_s)\in\mathfrak{gl}(\mathfrak{g})$
is semi-simple with only imaginary eigenvalues;
\item The linear endomorphism
$\mathrm{ad}(X_r)\in\mathfrak{gl}(\mathfrak{g})$ is nilpotent, i.e. it has only zero eigenvalues;
\item These two endomorphisms commute because
$[X_r,X_s]=0$.
\item By a suitable conjugation, we can present
$\mathrm{ad}(X)$, $\mathrm{ad}({X_r})$ and $\mathrm{ad}({X_s})$ as upper triangular, strict upper triangular and diagonal matrices respectively.
So $\mathrm{ad}(X)\in\mathfrak{gl}(\mathfrak{g})$ has the same
    eigenvalues (counting multiples) as $\mathrm{ad}(X_s)$.
\item The centralizer $\mathfrak{c}_{\mathfrak{s}_c}(\mathfrak{r}(\mathfrak{g}))$ containing $X_s$ is an compact semi-simple ideal
of $\mathfrak{g}$ contained in the intersection of all Levi subalgebras.
\end{enumerate}

The observations (1)--(3)
 imply $\mathrm{ad}(X)=\mathrm{ad}(X_s)
+\mathrm{ad}(X_r)$ is a Jordan--Chevalley decomposition, and hence
$X=X_s+X_r$ is an abstract Jordan decomposition. See 4.2 and 5.4 in \cite{Hu1972} for a comprehensive discussion of these notions.

The observation (4) explains why $\mathrm{ad}(X)$ has only imaginary eigenvalues, which solves our spectral problem
for bounded Killing vector fields.

Notice that the decomposition $\mathrm{ad}(X)=\mathrm{ad}(X_r)+\mathrm{ad}(X_s)$ is unique by
the uniqueness of Jordan--Chevalley decomposition, while the abstract Jordan decomposition may not be because of the
center $\mathfrak{c}(\mathfrak{g})$. However, by the observation (5), the decomposition
$X=X_r+X_s$ is unique in the sense that it does not depends on
the choice of the Levi subalgebra.

Above observations and discussions can be summarized to the following theorem.

\begin{theorem}\label{main-cor}
Let $M$ be a connected Riemannian manifold on which the connected Lie group $G$ acts effectively and isometrically.
Assume that $X\in\mathfrak{g}$ defines a bounded Killing vector field. Let $X$ be decomposed as $X=X_r+X_s$  according to any Levi
decomposition $\mathfrak{g}=\mathfrak{r}(\mathfrak{g})+\mathfrak{s}$, then we have the following:
\begin{enumerate}
\item The decomposition $\mathrm{ad}(X)=\mathrm{ad}(X_r)+\mathrm{ad}(X_s)$ is the unique Jordan--Chevalley decomposition for $\mathrm{ad}(X)$ in
    $\mathfrak{gl}(\mathfrak{g})$;
\item The decomposition $X=X_r+X_s$ is the abstract
Jordan decomposition which is unique in the sense that $X_s$ is
contained in all Levi subalgebras, i.e. this decomposition is
irrelevant to the choice of the Levi subalgebra;
\item The eigenvalues of $\mathrm{ad}(X)$ coincide with
those of those of $\mathrm{ad}(X_s)$, counting multiples.
\end{enumerate}
\end{theorem}

\newpage

\subsection{Bounded Killing vectors on a connected Riemannian homogeneous space}

In this section, we will always assume that
$M=G/H$ is a Riemannian homogeneous space on which the connected Lie group $G$ acts effectively. Applying Theorem \ref{theorem 2} and some argument in its proof, we can completely determine all the bounded vectors for $G/H$ as following.

Let $\mathfrak{g}=\mathfrak{r}(\mathfrak{g})+\mathfrak{s}$ be
a Levi decomposition, and $\mathfrak{s}=\mathfrak{s}_{c}\oplus
\mathfrak{s}_{nc}$ be a Lie algebra direct sum decomposition.
By Lemma \ref{lemma 0}, we have a reductive decomposition
$\mathfrak{g}=\mathfrak{h}+\mathfrak{m}$ such that the nilradical
$\mathfrak{n}(\mathfrak{g})$ is contained in $\mathfrak{m}$.

We have mentioned that $\mathfrak{c}_{\mathfrak{s}_c}
(\mathfrak{r}(\mathfrak{g}))$ is an ideal of $\mathfrak{g}$
contained in $\mathfrak{s}_c$. Denote $\mathfrak{s}'_c$ the ideal
of $\mathfrak{s}_c$ such that $\mathfrak{s}_c=\mathfrak{c}_{\mathfrak{s}_c}
(\mathfrak{r}(\mathfrak{g}))\oplus\mathfrak{s}'_c$. Then
we have a Lie algebra direct sum decomposition
\begin{equation}\label{008}
\mathfrak{g}=\mathfrak{c}_{\mathfrak{s}_c}
(\mathfrak{r}(\mathfrak{g}))\oplus (\mathfrak{s}'_c+\mathfrak{s}_{nc}+\mathfrak{r}(\mathfrak{g})).
\end{equation}
By this observation, we find the following lemma.

\begin{lemma}\label{lemma 6}
Keep all assumptions and notations of this section, then
any vector in~$\mathfrak{c}_{\mathfrak{s}_c}
(\mathfrak{r}(\mathfrak{g}))$ is bounded for $G/H$.
\end{lemma}
\begin{proof}
The ideal
$\mathfrak{c}_{\mathfrak{s}_c}(\mathfrak{r}(\mathfrak{g}))$
generates a compact semi-simple subgroup in $G$. So for any
$X\in\mathfrak{c}_{\mathfrak{s}_c}(\mathfrak{r}(\mathfrak{g}))$,
the orbit
$$
\mathrm{Ad}(G)X=\mathrm{Ad}\bigl(\exp\mathfrak{c}_{\mathfrak{s}_c}
(\mathfrak{r}(\mathfrak{g}))\bigr)
$$ is a compact set, which projection
in $\mathfrak{g}/\mathfrak{h}$ is obviously bounded with respect
to any norm. So any vector $X\in\mathfrak{c}_{\mathfrak{s}_c}
(\mathfrak{r}(\mathfrak{g}))$ is bounded for $G/H$, which proves this lemma.
\end{proof}
\smallskip

Obviously linear combinations of bounded vectors for $G/H$ are still bounded vectors for $G/H$, i.e. the set of all bounded vectors for $G/H$ is a real linear subspace of $\mathfrak{g}$. It is preserved by all
$\mathrm{Ad}(G)$-actions. So it is an ideal of $\mathfrak{g}$.
Applying Theorem \ref{theorem 2} and Lemma \ref{lemma 6}, we get the following immediate consequence.

\begin{theorem}\label{main-cor-2}
Assume $G/H$ is a Riemannian homogeneous space on which the connected Lie group $G$ acts effectively. Then the space of
all bounded vectors for $G/H$ is a compact ideal of $\mathfrak{g}$. Its
semi-simple part is coincides with
$\mathfrak{c}_{\mathfrak{s}_c}(\mathfrak{r}(\mathfrak{g}))$.
Its Abelian part~$\mathfrak{v}$ is contained in $\mathfrak{c}(\mathfrak{n}(\mathfrak{g}))$.
\end{theorem}

Before we continue to determine all the bounded vectors, there are several remarks.

For some Riemannian homogeneous spaces, bounded Killing vector fields can only be found from $\mathfrak{c}(\mathfrak{n}(\mathfrak{g}))$. For example,

\begin{corollary}
Let $G/H$ be a Riemannian homogeneous space which is diffeomorphic to an Euclidean space on which the connected Lie group $G$ acts effectively. Assume that $X\in\mathfrak{g}$ defines a bounded Killing vector field, then $X\in\mathfrak{c}(\mathfrak{n}(\mathfrak{g}))$.
\end{corollary}

\begin{proof}
Since $G/H$ is diffeomorphic to an Euclidean space,
the subgroup $H$ is
a maximal compact subgroup of $G$. Assume conversely that $X$
is not contained in $\mathfrak{c}
(\mathfrak{n}(\mathfrak{g}))$, then by Theorem \ref{theorem 2}
or Theorem \ref{main-cor-2}, there exists a non-trivial compact semi-simple normal subgroup $H'$ of $G$. We can get $H'\subset H$
by the conjugation theorem for maximal compact subgroups (i.e.
Theorem 14.1.3 in \cite{JK}). This is a contradiction to the effectiveness of the $G$-action.
\end{proof}

When $G/H$ is a geodesic orbit space (that means that every geodesic is an orbit of some one-parameter isometry group from $G$),
the second author have proved that any vector in
$\mathfrak{c}(\mathfrak{n}(\mathfrak{g}))$ defines a Killing vector field of constant length (see Theorem 1 in \cite{Ni2013} or Theorem 5 in \cite{Ni2019}).
By Theorem \ref{main-cor-2}, it implies an equivalence between the boundedness and the constant length condition
for Killing vector fields $X\in\mathfrak{n}(\mathfrak{g})$ for a~geodesic orbit space. Similar phenomenon can also be seen from Corollary 3.4 in
\cite{Wo2017}, for exponential solvable Lie groups endowed with
left invariant metrics.
\smallskip

Applying a
similar style for defining the restrictive Clifford--Wolf homogeneity \cite{BN2009} and the $\delta$-homogeneity
(which is equivalent to the notion of the generalized normal homogeneity) \cite{BN2008,BN2014}, we can use bounded Killing vector fields in order to
define the following condition for Riemannian homogeneous spaces.

\begin{definition}
Let $G/H$ be a Riemannian homogeneous space on which the connected Lie group $G$ acts effectively. Then it satisfies Condition {\rm(BH)} if for
any $x\in G/H$ and any
$v\in T_x(G/H)$, there exists a bounded vector $X\in\mathfrak{g}$
such that $X(x)=v$.
\end{definition}

Then Theorem \ref{main-cor-2} provides the following criterion for Condition (BH).

\begin{corollary}
Let
$G/H$ be a Riemannian homogeneous space on which the connected Lie group $G$ acts
effectively. Then it satisfies Condition {\rm(BH)} iff there exists a connected subgroup $K$ of $G$ such that
its Lie algebra is compact and the $K$-action on $G/H$ is transitive.
\end{corollary}

\begin{proof}
If $G/H$ satisfies Condition (BH), then
the space of all bounded vectors for $G/H$ generated a
connected quasi-compact subgroup $K$ of $G$ which acts transitively
on $G/H$.

Conversely, if such a quasi-compact subgroup exists,
all vectors in it are bounded for $G/H$. The Condition (BH) is satisfied because the exponential map from $\mathrm{Lie}(K)$ to $K$ is surjective.

This completes the proof of the corollary.
\end{proof}
\smallskip

To completely determine all bounded vectors for $G/H$, we just need to determine the subspace $\mathfrak{v}$ of
all bounded vectors $X\in\mathfrak{c}(\mathfrak{n})$ for $G/H$. Obviously the $\mathrm{Ad}(G)$-actions preserve $\mathfrak{v}$,
which is contained in the summand $\mathfrak{m}$ in the reductive decomposition. The condition that
$X\in\mathfrak{v}$, i.e. $X\in\mathfrak{c}(\mathfrak{n})$ is bounded for $G/H$, is
equivalent to that $\mathrm{Ad}(G)X$ is a bounded set in $\mathfrak{c}(\mathfrak{n})$ with respect to any norm.

The restriction of the $\mathrm{Ad}(G)$-actions defines
a Lie group endomorphism $\mathrm{Ad}_\mathfrak{v}$ from $G/N$ to the general linear group
$\mathrm{GL}(\mathfrak{v})$, where $N$ is the closed connected normal subgroup generated by $\mathfrak{n}(\mathfrak{g})$.
The tangent map at $e$ for $\mathrm{Ad}_\mathfrak{v}$ induces
the Lie algebra endomorphism which coincides with
$\mathrm{ad}_\mathfrak{v}$ defined by restricting the $\mathrm{ad}$-action from $\mathfrak{v}$ to~$\mathfrak{v}$.

The following key lemma helps us determine the subspace
$\mathfrak{v}$.

\begin{lemma}\label{lemma 7}
Let $G/H$ be a Riemannian homogeneous space on which the connected Lie group $G$ acts effectively. Keep all relevant assumptions and notations.
Then the image $\mathrm{Ad}_{\mathfrak{v}}(G)$ has a compact closure in $\mathrm{GL}(\mathfrak{v})$.
\end{lemma}
\begin{proof}
Fix a quadratic norm $\|\cdot\|=\langle\cdot,\cdot\rangle^{1/2}$ on $\mathfrak{v}$ and
an orthonormal basis $\{v_1,\ldots,v_k\}$ for $\mathfrak{v}$.
By the boundedness of each $v_i$ for $G/H$, and the speciality of
the reductive decomposition, we can find a positive $c_i>0$, such
that
$$
\|\mathrm{Ad}(g)v_i\|<c_i,\quad\forall\, g\in G,\quad\forall\, i=1,\ldots,k.
$$
For any $v\in\mathfrak{v}$ with $||v||=1$, we can present it
as $v=\sum_{i=1}^k a_iv_i$ with $\sum_{k=1}^k {a_i}^2=1$,
then for any $g\in G$ we have
$$
\|\mathrm{Ad}(g)v\|\leq\sum_{i=1}^k|a_i|\cdot \|\mathrm{Ad}(g)v_i\|\leq
C=c_1+\cdots+c_k.
$$
So we get
\begin{equation}\label{009}
C^{-1}\|v\|\leq \|\mathrm{Ad}(g)v\|\leq C \|v \|,
\quad\forall\, g\in G,\quad\forall\, v\in\mathfrak{v}.
\end{equation}

For
any sequence $\mathrm{Ad}_\mathfrak{v}(g_i)$ with $g_i\in G$, we can find a subsequence $\mathrm{Ad}_\mathfrak{v}(g'_i)$ such that \linebreak
$\lim\limits_{i\rightarrow\infty}\mathrm{Ad}_\mathfrak{v}(g'_i)v_j$ exists for each $j$,
so $\mathrm{Ad}_\mathfrak{v}(g'_i)$ converges to a $\mathbb{R}$-linear endomorphism $A$.
By continuity, the estimates (\ref{009}) for each
$\mathrm{Ad}_\mathfrak{v}(g_i)$ can be inherited by $A$, from which we see that $A\in \mathrm{GL}(\mathfrak{v})$.
So
$\mathrm{Ad}_\mathfrak{v}(G)$ has a compact closure in $\mathrm{GL}(\mathfrak{v})$, which
proves this lemma.
\end{proof}

By Lemma \ref{lemma 7},
$\mathrm{ad}_\mathfrak{v}$ maps the reductive Lie algebra
$\mathrm{Lie}(G/N)=\mathfrak{s}_c\oplus\mathfrak{s}_{nc}
\oplus\mathfrak{a}$, where $\mathfrak{a}=
\mathfrak{r}/\mathfrak{n}$, to a compact subalgebra. The summand
$\mathfrak{s}_{nc}$ must be mapped to 0, from which we get $[\mathfrak{s}_{nc},\mathfrak{v}]=0$.
Then it is easy to see
that
$$
\mathfrak{v}\subset
\mathfrak{c}_{\mathfrak{c}(\mathfrak{n})}
(\mathfrak{s}_{nc})\quad\mbox{and}\quad
[\mathfrak{r}(\mathfrak{g}),\mathfrak{c}_{\mathfrak{c}(\mathfrak{n})}
(\mathfrak{s}_{nc})]
\subset\mathfrak{c}_{\mathfrak{c}
(\mathfrak{n})}(\mathfrak{s}_{nc}).
$$

Moreover, the Abelian summand $\mathfrak{a}$ in $\mathrm{Lie}(G/N)$ is mapped to a space of semi-simple
matrices with imaginary eigenvalues, so $\mathfrak{v}$ can be
decomposed as a sum $$\mathfrak{v}=\mathfrak{v}_1\oplus\cdots\oplus\mathfrak{v}_k$$
of irreducible representations of
$\mathfrak{a}$, each of which is either one-dimensional or two-dimensional. Any one-dimensional $\mathfrak{v}_i$
must be a
trivial representation of $\mathfrak{a}$, and hence $\mathfrak{v}_i\subset
\mathfrak{c}_{\mathfrak{c}(\mathfrak{r}(\mathfrak{g}))}
(\mathfrak{s}_{nc})$.
Any two-dimensional $\mathfrak{v}_i$ corresponds to a pair of imaginary weights in
$\mathfrak{a}^*\otimes\mathbb{C}$, i.e. $\mathbb{R}$-linear functionals
$\pm\lambda:\mathfrak{a}\rightarrow\mathbb{R}\sqrt{-1}$, such that
the eigenvalues of $\mathrm{ad}_\mathfrak{v}(u):\mathfrak{v}_i\rightarrow\mathfrak{v}_i$ are $\pm\lambda(u)$.

Conversely, we consider the sum $\mathfrak{v}'$
of the centralizer $\mathfrak{c}_{\mathfrak{c}(\mathfrak{r}(\mathfrak{g}))}
(\mathfrak{s}_{nc})$ and all two-dimensio\-nal irreducible
$\mathrm{ad}(\mathfrak{r})$-representations in
$ \mathfrak{c}_{\mathfrak{c}(\mathfrak{n}(\mathfrak{g}))}
(\mathfrak{s}_{nc})$
corresponding to imaginary weights of
$\mathfrak{a}=\mathfrak{r}
(\mathfrak{g})/\mathfrak{n}(\mathfrak{g})$.
Then $\mathfrak{v}'$ is $\mathrm{Ad}(G)$-invariant, and $\mathfrak{v}\subset\mathfrak{v}'$. Denote
 $\mathrm{Ad}_{\mathfrak{v}'}$ the restriction
of the $\mathrm{Ad}_\mathfrak{g}$-action from $\mathfrak{v}'$ to $\mathfrak{v}'$.
The subspace
$\mathfrak{v}'$ satisfies similar descriptions
as given above for $\mathfrak{v}$.
The image group $\mathrm{Ad}_{\mathfrak{v}'}(R/N)$ is contained in a torus which commutes with
the image $\mathrm{Ad}_{\mathfrak{v}'}(S_c)\subset\mathrm{GL}(\mathfrak{v}')$ of the compact subgroup $S_c=\exp\mathfrak{s}_c$,  so
$$
\mathrm{Ad}_{\mathfrak{v}'}(G)=
\mathrm{Ad}_{\mathfrak{v}'}(S_c)\cdot\mathrm{Ad}_{\mathfrak{v}'}
(R/N)\subset\mathrm{Ad}_{\mathfrak{v}'}(S_c)\cdot
\overline{\mathrm{Ad}_{\mathfrak{v}'}
(R/N)}
$$
has a compact closure in~$\mathrm{GL}(\mathfrak{v}')$.
This implies all vectors $X\in\mathfrak{v}'$ are bounded for $G/H$, i.e. $\mathfrak{v}=\mathfrak{v}'$.

Summarizing above argument, we get the following theorem which determines all bounded vectors $X\in\mathfrak{c}
(\mathfrak{n}(\mathfrak{g}))$ for a Riemannian homogeneous space
$G/H$.

\begin{theorem}\label{main-cor-3}
Let $G/H$ be a Riemannian homogeneous space on which the connected Lie group $G$ acts effectively.
Keep all relevant assumptions and notations. Then the space $\mathfrak{v}$ of
all bounded vectors $X\in\mathfrak{c}(\mathfrak{n})$ for $G/H$ is the sum of
$\mathfrak{c}_{\mathfrak{c}(\mathfrak{r}(\mathfrak{g}))}
(\mathfrak{s}_{nc})$ and all two-dimensional
irreducible representations in
$\mathfrak{c}_{\mathfrak{c}(\mathfrak{n})}(\mathfrak{s}_{nc})$ for the $\mathrm{ad}(\mathfrak{r}(\mathfrak{g}))$-actions,
which corresponds to nonzero imaginary weights in $\mathfrak{a}^*\otimes\mathbb{C}$,
i.e. nonzero $\mathbb{R}$-linear functionals $\lambda:\mathfrak{r}(\mathfrak{g})\rightarrow
\mathfrak{r}(\mathfrak{g})/
\mathfrak{n}(\mathfrak{g})\rightarrow\mathbb{R}\sqrt{-1}$.
\end{theorem}

A theoretical algorithm presenting all vectors in $\mathfrak{v}$ can be given as follows. Let us consider
the complex representation for $\mathrm{ad}(\mathfrak{r}(\mathfrak{g}))$-actions on
$\mathfrak{c}_{\mathfrak{c}(\mathfrak{n})}(\mathfrak{s}_{nc})\otimes\mathbb{C}$, then we can find distinct real weights
$\lambda_i\in\mathfrak{a}^*
\subset\mathfrak{r}^*$, $1\leq i\leq n_1$, and non-real complex weights $a_j\pm b_j\sqrt{-1}\in\mathfrak{a}^*\otimes\mathbb{C}
\subset\mathfrak{r}^*\otimes\mathbb{C}$, with $1\leq i\leq n_2$ and $b_j>0$, such that
we have the direct sum decomposition
$$
\mathfrak{c}_{\mathfrak{c}(\mathfrak{n})}(\mathfrak{s}_{nc})\otimes\mathbb{C}=
\bigoplus_{i=1}^{n_1}\mathfrak{u}^\mathbb{C}_{\lambda_i}
\,\oplus\,\bigoplus_{j=1}^{n_2}\left(
\mathfrak{u}^{\mathbb{C}}_{a_j+b_j\sqrt{-1}}+
\mathfrak{u}^{\mathbb{C}}_{a_j-b_j\sqrt{-1}}\right),
$$
where the complex subspace
$\mathfrak{u}^{\mathbb{C}}_\alpha$ for any real weight $\alpha\in\mathfrak{a}^*$ or any complex weight $\alpha\in\mathfrak{a}^*\otimes\mathbb{C}$ is defined as
$$
\mathfrak{u}^\mathbb{C}_\alpha
=\left\{X\in\mathfrak{c}_{\mathfrak{c}(\mathfrak{n})}
(\mathfrak{s}_{nc})\otimes\mathbb{C}\,|\,
\bigl(\mathrm{ad}(u)-\alpha(u)\mathrm{Id}\bigr)^k X=0,\forall\, u\in\mathfrak{r}(\mathfrak{g}), \mbox{ for some }k>0\right\}.
$$

We assume $a_j=0\in\mathfrak{a}^*$ iff $1\leq j\leq m$, where $m$ can be zero.
Denote $\mathfrak{v}^{\mathbb{C}}_{\alpha}$ be the eigenvector subspace in $\mathfrak{u}^{\mathbb{C}}_{\alpha}$ for the
$\mathrm{ad}(\mathfrak{r}(\mathfrak{g}))$-actions, i.e.
$$
\mathfrak{v}^\mathbb{C}_{\alpha}=
\left\{X\in\mathfrak{u}^{\mathbb{C}}_{\alpha}\,|\,
\mathrm{ad}(u)X=\alpha(u) X,\forall \,u\in\mathfrak{r}(\mathfrak{g})\right\}.
$$
We take any basis $\left\{v_{j,1},\ldots,v_{j,k_j}\right\}$ for $\mathfrak{v}^\mathbb{C}_{b_j\sqrt{-1}}$\,, then
the space $\mathfrak{v}$ of all bounded vectors in
$\mathfrak{c}(\mathfrak{n})$ for $G/H$ can be presented as
\begin{eqnarray*}
\mathfrak{v}&=&\left(\mathfrak{v}^{\mathbb{C}}_0\cap\mathfrak{v}\right)
\oplus
\bigoplus_{j=1}^{m}\left(\left(\mathfrak{v}^\mathbb{C}_{b_j\sqrt{-1}}+
\mathfrak{v}^{\mathbb{C}}_{-b_j\sqrt{-1}}\right)\cap\mathfrak{v}\right)\\
&=&\mathfrak{c}_{\mathfrak{c}(\mathfrak{r}(\mathfrak{g}))}
(\mathfrak{s}_{nc})\oplus
\mathrm{span}^\mathbb{R} \left\{v_{j,k}+\overline{v_{j,k}},
\sqrt{-1}\bigl(v_{j,k}-\overline{v_{j,k}}\bigr),\forall\, 1\leq j\leq m, 1\leq k\leq k_j\right\}.
\end{eqnarray*}
\smallskip
We hope that all the above results will be useful in the study of related topics.
\medskip

{\bf Acknowledgements.} The authors would like to sincerely thank the Chern Institute and the School of Mathematical Sciences in Nankai University for their hospitality during the preparation of
this paper. They would also sincerely thank Joseph A. Wolf for reading this paper and helpful discussions. The first author is supported by National Natural Science Foundation of China (No. 11821101, No. 11771331), Beijing Natural Science Foundation
(No. 00719210010001, No. 1182006), Capacity Building for Sci-Tech  Innovation -- Fundamental Scientific Research Funds (No. KM201910028021).

\vspace{5mm}
\end{document}